\documentclass{article}

\usepackage{amsfonts}
\usepackage{amssymb}
\usepackage{amsmath}
\usepackage{amsthm}
\usepackage{comment}
\usepackage[all]{xy}
\usepackage{color}

\theoremstyle{plain}
\newtheorem{theorem}{Theorem}[section]
\newtheorem{lemma}[theorem]{Lemma}
\newtheorem{prop}[theorem]{Proposition}
\newtheorem{corollary}[theorem]{Corollary}

\theoremstyle{definition}

\newtheorem{example}[theorem]{Example}

\def\es{\emptyset}
\def\R{\mathbb{R}}
\def\C{\mathbb{C}}
\def\N{\mathbb{N}}
\def\VE{V(E)}
\def\intt{\int \limits_{-\infty}^{\infty}}
\def\cl{{\rm cl}}
\def\L{\mathcal{L}}
\def\D{\mathcal{D}}
\def\ep{\varepsilon}

\newcommand{\supp}{{\rm supp}\,}
\newcommand{\lin}{{\rm span}\,}

\title{\bf On density of shift-invariant subspaces of some inductive limit spaces}
\author{J\'ozef Burzyk}
\date{}

\begin{document}
\maketitle

\begin{abstract}
Necessary and sufficient conditions are given for density of shift-invariant subspaces of the space
$\mathcal{L}$ of integrable functions of bounded support with the inductive limit topology.
\end{abstract}



\section{Introduction}
A space $E$ of functions defined on the real line $\R$ is shift-invariant if $f \in E$ implies
$\tau_{\lambda}f \in E \ $ for all $\lambda \in \R$, where $\tau_{\lambda} f(t)=f(t-\lambda)$. A theorem of Wiener (see \cite{Wiener}) says that a shift invariant subspace is related to density of shift-invariant subspaces of the space of integrable functions on the real line. In this paper we are concerned with density of shift-invariant subspaces of the space of integrable functions of bounded supports equipped with the natural inductive limit topology.

The proofs of the main theorems (Theorems 1 and 2) are  based on Corollary 1 in \cite{BUR2} and involve the Fourier transforms of regular Mikusi\'{n}ski operators of bounded support. Using Theorems 1 and 2 we obtain characterizations of the so-called total sets as well as of Mikusi\'nski operators whose denominators are of bounded support. Then we discuss possible generalizations of Theorems 1 and 2 and give an example of a topologically dense shift-invariant subspace of $\L$ which is not sequentially dense.

Now we need to introduce some notation. By $\mathcal{L}$ we denote the space of integrable functions on the real line $\R$ with the bounded support. For a function $\varphi \in \mathcal{L}$ by $\supp(\varphi)$
we denote the support of $\varphi$, i.e., the smallest closed set $S$ such that $\int_{\mathbb{R} \setminus S}|\varphi(t)|dt =0$.

For a $K \subset \R$ by $\mathcal{L}_K$ we denote the set of
all functions $\varphi \in \mathcal{L}$ such that $\supp(\varphi) \subset K$ and for any positive
$\alpha$ by $\mathcal{L}_{\alpha}$ we denote the set $\mathcal{L}_{[-\alpha, \alpha]}$.

For a function $\varphi \in \mathcal{L}$ we denote $\|\varphi\|=\displaystyle \int_{\R} |\varphi(t)|\,
dt$ and we equip $\mathcal{L}$ with the strict inductive limit topology of the Banach spaces
$(\mathcal{L}_n,\|\cdot\|)$, $n \in \N$.

For  $A\subset \L$, by $\lambda(A) $ we denote the sequential closure of $A$ (i.e., $\varphi \in
\lambda(A)$ if and only if $\varphi_n \to \varphi$ in $\mathcal{L}$ for some $\varphi_n \in A$), and
by $\cl(A)$ the topological closure of $A$.

If $\varphi \in \L$, then the Fourier transform
$$\hat{\varphi}(z)=\intt e^{izt} \varphi(t)dt$$
is an entire functions. We let
$$Z(\varphi)=\{ z \in \C: \ \widehat{\varphi}(z)=0 \}$$
and
$$V(\varphi) = \{ (z,n) \in \C \times \N: \widehat{\varphi}^{(k)}(z)=0
\mbox{ for }  k=0,1,\dots ,n \}.$$
For $E\subset \mathcal{L}$ we define
$$Z(E) = \bigcap_{\varphi \in E}  Z(\varphi)\quad \mbox{and} \quad V(E) = \bigcap_{\varphi \in E} V(\varphi).$$

A sequence $(\delta_n)$ of integrable functions on $\R$ is said to be a weak delta sequence if the
following conditions hold:

\begin{itemize}
\item [(i)]  $\supp ({\delta_n}) \subset [-\alpha,\alpha] \ $ for some  $\alpha>0$ and all $\ n \in \N$;
\item [(ii)] $\displaystyle \lim_{n\to\infty} \int_{\R}\delta_n(t)\, dt = 1 $;
\item [(iii)] $\displaystyle \int_{\R}|\delta_n(t)|\, dt < M $ for  some $M < \infty$ and all $ n \in \N$;
\item [(iv)] $\displaystyle \lim_{n\to\infty} \int_{|t|>\ep}|\delta_n(t)|\,dt = 0$ for every $\ep > 0$.
\end{itemize}

These delta sequences are called weak to distinguish them from standard delta sequences, for which
conditions (i) and (iv) are replaced by the condition $\supp (\delta_n) \subset [- {\ep}_n,{\ep}_n]$
with ${\ep}_n \to 0$.

It can be proved that a sequence $(\delta_n)$ is a weak delta sequence if and only if for each
locally integrable function $f$ and for each $\alpha>0$ we have:
$$
\lim_{n \to \infty}\int_{-\alpha}^{\alpha}|f*\delta_n(t)-f(t)|dt=0.
$$

By $L_+$ we denote the ring of locally integrable on the real line vanishing to the left of some
point of the real line. The ring operations in $L_+$ are the usual addition and the convolution as
the product operation. By the Titchmarsh theorem $L_+$ is a ring without zero divisors. The field
of Mikusi\'nski operators is defined as a field of quotients of this ring (see \cite{MIK}).

For a function $\varphi \in L_+$ we define
$$
\Lambda(\varphi)=\sup\{\lambda:\ f=0 \ \textrm{a.e. on } \ (-\infty, \lambda)\}
$$
and additionally, if $\varphi \in \mathcal{L}$, then
$$
\Gamma(\varphi)=\inf\{\lambda:\ f=0 \ \textrm{a.e. on } \ (\lambda, \infty).\}
$$
By the Titchmarsh theorem
\begin{equation}
\label{TITL} \Lambda(\varphi*\psi)=\Lambda(\varphi)+\Lambda(\psi) \quad \textrm{for any} \quad
\varphi, \ \psi \in L_+
\end{equation}
and consequently
\begin{equation}
\label{TITG} \Gamma(\varphi*\psi)=\Gamma(\varphi)+\Gamma(\psi) \quad \textrm{for any} \quad
\varphi, \ \psi \in \mathcal{L}.
\end{equation}

If $\xi=\displaystyle \frac {\varphi} {\psi}$ is a Mikusi\'nski operator, we define
$$
\Lambda(\xi) = \Lambda(\varphi)-\Lambda(\psi).
$$
If, additionally, an operator $\xi$ have a representation $\xi=\displaystyle \frac {\varphi} {\psi}$,
where $\varphi, \psi \in \mathcal{L}$, then
$$
\Gamma(\xi)=\Gamma(\varphi)-\Gamma(\psi).
$$
By (\ref{TITL}) and (\ref{TITG}) both definitions are correct.

For a locally integrable function $\varphi$ on $\R$ we define the function
$\widetilde{\varphi}(t)=\varphi(-t)$, and for any $\lambda \in \R$ let $\tau_{\lambda}\varphi(t)=
\varphi(t-\lambda)$. If $\xi=\displaystyle \frac {\varphi} {\psi}$  is a Mikusi\'nski operator, then
$\tau_{\lambda}\xi=\displaystyle \frac {\tau_{\lambda}\varphi} {\psi}$ for any $\lambda \in \R$.

By $s$ we denote the differential operator, that is, $s=\displaystyle \frac 1 l$, where $l$ is the
characteristic function of $[0, \infty)$. For each complex number $\alpha$ we have $\frac 1 {s-\alpha}=e_{\alpha}$,
where $e_{\alpha}$ is the function
$$
e_{\alpha}(t)=
\begin{cases}
e^{\alpha t} & \text{if } t \geq 0,\\
0 & \text{if } t<0.
\end{cases}
$$
We will use the following proposition.

\begin{prop}
\label{C1} Suppose $\varphi \in \L_{[a,b]}$ and $(z_0,k) \in V(\varphi)$ and let 
\begin{equation}
\label{WND} \psi_k = (-1)^k i^k \frac{\varphi}{(s+iz_0)^k}.
\end{equation}
Then, $\psi_k \in \L_{[a,b]}$ and $\widehat{\psi}_k(z) = \frac{\widehat{\varphi}(z)}{(z-z_0)^k}$.
\end{prop}

\begin{proof}
Since $\frac{1}{(s+iz_0)^k} = g_k$, where $g_k(x) = \frac{x^{k-1}}{(k-1)!} \, e^{-iz_0x}$
for $x \ge 0$, the proof follows by induction. 
\end{proof}

Following Boehme \cite{BOM}, an operator $x$ is said to be
regular if it admits a representation
$
x=\frac {f_n} {\delta_n}
$
for all $n \in \N$, where $f_n \in \mathcal{L}_{+}$ and $(\delta_n)$
is a delta-sequence. We say that a regular operator $x$ is zero on an open set $\Omega \subset \R$ if $x$ admits a representation 
$
x=\frac {f_n} {\delta_n}
$
such that $(\delta_n)$ is a delta-sequence and $f_n \in \mathcal{L}_{+}$ is sequence that converges to $0$ uniformly on every compact subset of $\Omega$.
The support of a regular operator $x$ is the complement of the largest open set on which $x$ is zero (see \cite{BOM}). Using the famous Beurling-Malliavin theorem (see \cite{BM}) it was proved in \cite{BUR2} that an operator $x$ is a regular operator with bounded support if and only if it has a representation $x=\frac {\varphi} {\psi}$, where $\varphi, \psi \in \mathcal{L}$ and $\widehat{\varphi}(z)/\widehat{\psi}(z)$ is an entire function.

\section{Bounded interval}

\begin{theorem}
\label{l1} Suppose $E$ is a subset of $\mathcal{L}$ such that
\begin{equation}
\label{ZAP} \inf \{ { \Lambda } (\varphi): {\varphi} \in E \}=0 \quad \text{and} \quad \sup \{ {\Gamma} (\varphi):
{\varphi} \in E \}=\alpha.
\end{equation}
Let $\beta \geq \alpha$ and let
$$
F=\lin \{ {\tau}_{\lambda} {\varphi}: \varphi \in E, \ \lambda \in [0, \beta]\}.
$$
Then, for any function $\varphi \in \mathcal{L}$ such that $\supp(\varphi) \subset [0, \alpha+\beta)$, we
have
$$
\varphi \in \cl(F) \quad \text{if and only if} \quad V(E) \subset V(\varphi).
$$
\end{theorem}

\begin{proof}
Suppose that there is a function $\varphi_0\in\mathcal{L}$ such that $\supp(\varphi_0) \subset [0, \alpha+\beta)$, $V(E) \subset V(\varphi_0)$, and $\varphi_0 \notin \cl(F)$. Then, by the Hahn-Banach theorem (see \cite{RUDIN}), there is a measurable and bounded function $u$ on $[0, \alpha+\beta]$  such that
\begin{equation}
\label{l1_2} \int_{0}^{\alpha+\beta}u(t){\varphi}_0(t)\,dt=1
\end{equation}
and
\begin{equation}
\label{l1_2b} \int_{0}^{\alpha+\beta} u(t) \varphi(t) dt=0
\end{equation}
for any $\varphi \in \cl(F)$.

Note that for each function $\varphi \in E$ the convolution $u*\widetilde{\varphi}$ is a continuous function
with support in $[-\alpha, \alpha+\beta]$. Moreover, it follows from (\ref{l1_2b})  that $u*\widetilde{\varphi}$ vanishes on the interval $[0, \beta]$.
Hence, we can write
\begin{equation}
\label{DA} u*\widetilde{\varphi}=\tau_{-\alpha}a_1(\varphi)+\tau_{\beta}a_2(\varphi)
\end{equation}
where $a_1(\varphi), a_2(\varphi)$ are continuous functions with supports contained in the interval $[0,
\alpha]$.

Fix some $\varphi,\psi \in E$. Since $(u*\widetilde{\varphi})* \widetilde {{\psi}}=(u*\widetilde {\psi})* \widetilde
{{\varphi}}$, then
\begin{equation}
\label{RPP}
\tau_{-\alpha}a_1(\varphi)*\widetilde{\psi}+\tau_{\beta}a_2(\varphi)*\widetilde{\psi}=
\tau_{-\alpha}a_1(\psi)*\widetilde{\varphi}+\tau_{\beta}a_2(\psi)*\widetilde{\varphi}.
\end{equation}
From the Titchmarsh theorem we get
$$
\Gamma(\tau_{-\alpha}a_1(\varphi)*\widetilde{\psi})=-\alpha+\Gamma(a_1(\varphi))-\Lambda(\psi)
\leq 0
$$
and
$$
\Lambda(\tau_{\beta}a_2(\varphi)*\widetilde{\psi})=\beta+\Lambda(a_2(\varphi))-\Gamma(\psi)
\geq \beta-\alpha.
$$
Because the same inequalities hold for both components of the right-hand side of the equality (\ref{RPP}), then the equalities
$$
a_1(\varphi)*\widetilde{\psi}=a_1(\psi)*\widetilde{\varphi} \quad \text{and} \quad
a_2(\psi)*\widetilde{\varphi}=a_2(\varphi)*\widetilde{\psi}
$$
hold for any functions $\varphi, \psi \in E$. Hence
$$
\frac {a_1(\varphi)}  {\widetilde{\varphi}}= \frac {a_1(\psi)} {\widetilde{\psi}} 
\quad \text{and} \quad
\frac {a_2(\varphi)}  {\widetilde{\varphi}} = \frac {a_2(\psi)} {\widetilde{\psi}}
$$
(as Mikusi\'nski operators). Thus, the operators
$$
\xi_1=\frac {a_1(\varphi)*\widetilde{\varphi_0}} {\widetilde{\varphi}} 
\quad \text{and} \quad
\xi_2=\frac {a_2(\varphi)*\widetilde{\varphi_0}} {\widetilde{\varphi}} 
$$
do not depend on $\varphi \in E \setminus \{0\} $ and we have
\begin{equation}
\label{l1_5} u*\widetilde{\varphi_0}={\tau}_{-\alpha}\xi_1+\tau_{\beta}\xi_2,
\end{equation}
by (\ref {DA}).
Note that the operators $\xi_1$ and $\xi_2$ are regular in view of $\VE \subset V(\varphi_0)$ and Corollary 1 in \cite {BUR2}. Moreover,
$$
\Gamma(\xi_1) \leq \alpha+\Lambda(\varphi) \quad \text{and}\quad \Lambda(\xi_2) \geq
-\Gamma(\varphi_0)+\Gamma(\varphi)
$$
for each $\varphi \in E$. Hence, from (\ref{ZAP}), we obtain
$$
\Gamma(\xi_1) \leq \alpha \quad \text{and} \quad \Lambda(\xi_2)>-\beta
$$
and therefore
$$\Gamma(\tau_{-\alpha}\xi_1) \leq 0 \quad \text{and} \quad \Lambda(\tau_{\beta}\xi_2)>0.
$$
Hence, by (\ref{l1_5}), $u*\widetilde{\varphi_0}$ is a continuous function that vanishes on some interval $(0,
\ep)$. In particular, $u*\widetilde{\varphi_0}(0)=0$, which contradicts (\ref {l1_2}).
\end{proof}

\begin{corollary}
\label{WZMT} Let $E$ be a subset of $\mathcal{L}$ such that
$$
a_0=\inf \{ { \Lambda } (\varphi): {\varphi} \in E \}
\quad \text{and} \quad b_0=\sup \{ {\Gamma} (\varphi):
{\varphi} \in E \}
$$
are finite. Suppose the numbers $a<a_0$ and $b>b_0$ are such that $b-a \geq 2(b_0-a_0)$ and let
$$F=\lin\{\tau_\lambda \varphi:\ \varphi \in E, \ \lambda \in [a-a_0, b-b_0]\}.$$
Then, for each function $\varphi \in \mathcal{L}$ such that $\supp(\varphi) \subset [a, b)$, we have
$$
\varphi \in \cl(F) \quad \text{if and only if} \quad V(E) \subset V(\varphi).
$$
\end{corollary}
\begin{proof}
Let
$$
E_0=\{\tau_{-a_0} \varphi: \varphi \in E\} \quad \text{and} \quad F_0={\rm span} \{\tau_{\lambda} \varphi: \lambda \in [0, b-a], \ \varphi \in E_0\}.
$$
Then
\begin{equation}
\label{WTP} \inf \{\Lambda(\varphi):\ \varphi \in E_0\}=0 \quad \text{and} \quad
\sup\{\Gamma(\varphi): \ \varphi \in E_0\}=b_0-a_0.
\end{equation}
Suppose that $\varphi \in \mathcal{L} $ is any function such that $\supp(\varphi) \subset [a, b)$
and $V(E) \subset V(\varphi)$. Then $\supp(\tau_{-a}\varphi) \subset [0, b-a)$, and, by Theorem \ref{l1} and (\ref{WTP}), we get $\tau_{-a}\varphi \in \cl(F_0)$, which is equivalent to $\varphi \in \cl(F)$.
\end{proof}

\begin{corollary}\label{cor}
Under the assumptions of Theorem \ref{l1}, if $Z(E)=\emptyset$, then
\begin{equation}
\label{ASNP} \cl(F)=\mathcal{L}_{[0, \alpha+\beta]}.
\end{equation}
Moreover, under the assumption of Corollary \ref{WZMT}, we have
\begin{equation}
\label{ASNPD}
\cl(F)=\mathcal{L}_{[a, b]}.
\end{equation}
\end{corollary}
\begin{proof}
It suffices to note that $Z(E)=\emptyset$ if and only if $V(E)=\emptyset$. Then \eqref{ASNP} follows from Theorem \ref{l1} and \eqref{ASNPD} follows from Corollary \ref{WZMT}.
\end{proof}

\bigskip

\section{Shift invariant spaces of inductive limit}

\begin{theorem}
\label{CDPN}
Let $E$ be a shift-invariant subspace of  $\mathcal{L}$. Then
$$
\lambda(E)=\{\varphi \in \mathcal{L}: V(E_{\alpha}) \subset V(\varphi) \ \text{for some } \ \alpha>0\}.
$$
\end{theorem}


\begin{theorem}
\label{tw1} Let $E$ be a shift-invariant subspace of $\mathcal{L}$. The following conditions are
equivalent:
\begin{itemize}
\item [{\rm (a)}] $E$ is sequentially dense in $\mathcal{L}$;
\item [{\rm (b)}] For every $\alpha>0$ there exists $\beta>0$ such that $\mathcal{L}_{\alpha} \subset \cl(E_{\beta})$;
\item [{\rm (c)}] There is a weak delta-sequence $({\delta}_n)$ in $E;$
\item [{\rm (d)}] There is a $\alpha >0$ such that $Z(E_{\alpha})=\emptyset$.
\end{itemize}
\end{theorem}

\begin{proof} Assume that (a) holds and fix $\alpha>0$.
We have
$$
\mathcal{L}_{\alpha}=\bigcup_{k=1}^{\infty}\left(\cl(E_k) \cap \mathcal{L}_{\alpha}
\right).
$$
For each $k \in \N$ the set $\cl(E_k) \cap \mathcal{L}_{\alpha}$  is a closed
subspace of $\mathcal{L}_{\alpha}$. By the Baire category theorem (see \cite{RUDIN}), there is a $k \in \N$ such that
$\cl(E_k) \cap \mathcal{L}_{\alpha}$ contains an open subset of
$\mathcal{L}_{\alpha}$. Consequently, $\mathcal{L}_{\alpha} \subset \cl (E_k)$.

Obviously, (b) implies (c) and because $\widehat{\delta_n}(z) \to 1$ for any complex number $z$, then (c) implies (d).

Finally from Theorem \ref{CDPN} follows implication (d) $\Rightarrow$ (a).
\end{proof}

\begin{lemma}
\label{JPW} Suppose that $E$ is shift invariant subspace of $\mathcal{L}$. If $E_{\alpha}
\neq \{0\}$ for some $\alpha>0$ then
$$
V((\lambda E)_{\alpha})=V(E).
$$
\end{lemma}
\begin{proof}
The inclusion $V(E) \subset V((\lambda E)_{\alpha})$ is obvious. To prove
inverse inclusion suppose that $(z_0, n) \in V((\lambda E)_{\alpha}) \setminus V(E)$. Then there
exits function $\varphi_0 \in (\lambda E)_{\alpha}$ such that $(z_0, n) \in V(\varphi_0)$.

Let $m$ be a maximal natural number such that $(z_0, m) \in V(E)$ if such number exist, and
take $m=0$ in other case. Because of our assumption $m<n$.

Let $k$ be the order of zero of $\widehat{\varphi_0}$ at $z_0$ and $\varphi \in \mathcal{L}_{\alpha}$
be a function such that
$$
\widehat{\varphi}(z)=\frac {\widehat{\varphi_0}(z)} {(z-z_0)^{k-m}}
$$
Note, that from Proposition \ref{C1} such function exists and
$$
V(\varphi)=V(\varphi_0) \setminus \{(z_0, j):  j=m+1, \dots, k\}.
$$
Because $(z_0, m+1) \notin V(E)$, then there exists $\beta>\alpha$ such that $(z_0, m+1) \not \in
V(E_{\beta})$ and for such $\beta$ we have inclusion $V(E_{\beta}) \subset V(\varphi)$.

Let
$$H=\lin
\{\tau_{\lambda}\varphi: \varphi \in E_{\beta}, \ \lambda \in [-\beta, \beta]\}.
$$
Then $H \subset E_{\alpha+\beta}$ and by Corollary \ref{WZMT} $\varphi \in \cl(H)$ hence in
particular $\varphi \in \lambda(E)$. But $\supp(\varphi) \subset [-\alpha, \alpha]$, so $\varphi \in
(\lambda E)_{\alpha}$, so we get a contradiction with the assumption $(z_0, n) \in V((\lambda
E)_{\alpha})$, since $\widehat{\varphi}^{(m)}(z_0) \neq 0$ and $m<n$.
\end{proof}

\begin{theorem}
\label{THOD} Suppose that $E$ is a shift invariant subspace of $\mathcal{L}$. Then
$$
\cl(E)=\lambda^2(E)=\{\varphi \in \mathcal{L}:\ V(E) \subset V(\varphi)\}.
$$
\end{theorem}
\begin{proof}
Denote
$$
F=\{\varphi \in \mathcal{L}:\ V(E) \subset V(\varphi)\}.
$$
Because the inclusions $\lambda^2(E) \subset \cl E \subset F$ are obvious, it is enough to
prove inclusion $F \subset \lambda^2(E)$. To do this, take any function $\varphi \in \mathcal{L}$
and suppose that $\supp(\varphi) \subset [-\alpha, \alpha]$. By Lemma \ref{JPW} we can assume that
$V(E)=V(\lambda(E)_{\alpha})$. Because $\lambda(E)$ is a shift invariant subspace of
$\mathcal{L}$, then by Corollary \ref{WZMT} $\varphi \in
\cl(\lambda(E)_{2\alpha+1})$, hence $\varphi \in
\lambda^2(E)$.
\end{proof}

As a consequence of Lemma \ref{JPW} and Theorem \ref{THOD} we obtain the following characterization
of the dense subspaces of $\mathcal{L}:$
\begin{theorem}\label{tw2}
Let $E$ be a shift-invariant subspace of $\mathcal{L}$. Then the following conditions are
equivalent:
\begin{itemize}
\item [{\rm (a)}] $E$ is dense in $\mathcal{L};$
\item [{\rm (b)}] $Z(E)=\emptyset$;
\item [{\rm (c)}] $E \neq \{0\}$ and if for some $\alpha>0$ there exists a non-zero function $\varphi \in
\mathcal{L}_{\alpha}$ then $Z((\lambda(E)_{\alpha})=\emptyset$.
\end{itemize}
\end{theorem}

From Theorem \ref{tw2} easy follows a positive  solution of a problem of \'A. Sz\'az (\cite{SZA}) concerning the so-called total sets in the space $\mathcal{D}$ of smooth functions on the real-line.

\bigskip

We complement the preceding results by giving an example of a shift-invariant subspace of $\L$ which is topologically dense, but not sequentially dense.

\begin{example} For an arbitrary sequence of functions $\varphi_k\in \L$, the space
\begin{equation}
\label{ex1} E=\lin\{\tau_{\lambda}\varphi_k:   \lambda \in \R,  k \in \N \}
\end{equation}
is shift-invariant. If
\begin{equation}
\label{ex2} \bigcap_{k=1}^{\infty}Z(\varphi_k)=\es
\end{equation}
then, by Theorem \ref {tw2}, $E$ is topologically dense in $\L$. We are going to construct a
sequence $(\varphi_k)$ such that condition (\ref{ex2}) is satisfied, but $E$ is not sequentially
dense in $\L$.

Let $\D$  denote the space of smooth functions with bounded support. Assume that $\varphi \in \D$,
$\varphi(t) \not= 0  $ for $ |t|<1$ and $\varphi(t)=0 \ $ for $|t|>1$. Let $\ z_1,z_2,... \ $ be
zeros of $\widehat{\varphi}$ of orders $n_1,n_2,\dots $, respectively. Define
\[
\varphi_k=\left(\frac {\varphi}  {y_k}\right)^k,
\]
where $y_k=(s-iz_k)^{n_k}$, $k \in \N$, where the powers are understood in the sense of
Mikusi\'nski operators. We have $\varphi_k \in \D, \ \supp \varphi_k=[-k,k] \ $ and by Proposition
\ref{C1} $Z(\varphi_k)=Z(\varphi) \setminus \{z_k\} \ $ which implies that $Z(\{\varphi_k: k \in
\N\})=\es$. Therefore the sequence has the required properties, and $E$ defined in (\ref {ex1}) is
topologically dense in $\L$.

We claim that $E$ is not sequentially dense in $\mathcal{L}$. To prove this it suffices to show
that
\begin{equation}
\label{ex3}
E_k \subset E^k
\end{equation}
where $E^k={\rm lin} \{\tau_{\lambda} \varphi_r: \ \lambda \in \R, \ r=1,...,k \}$. This follows
from Theorem \ref {tw1}, because \eqref{ex3} implies $Z(E_k) \supset Z(E^k)=\{z_1,\dots ,z_k\}
\not=\es$ for all $k \in \N$.

To show (\ref{ex3}) assume that $ \psi \in E_k \cap E^m \ $
for some $m>k$. Then
\begin{equation}
\label{ex4}
\psi=\psi_1+\psi_2
\end{equation}
where $\psi_1 \in E^k$ and $\psi_2 \in \lin \{\tau_{\lambda} \varphi_r:  r=k+1,\dots ,m,  \lambda
\in \R \}$. Convolving both sides of (\ref {ex4}) with $\ y=y_1*\dots *y_m$ we obtain
\begin{equation}
\label{ex5} y*\psi=a_1*\varphi+a_2*\varphi^2+\dots+a_k*\varphi^k+\gamma*\varphi^k,
\end{equation}
where $\gamma \in \D$ and
\[
a_r \in \lin \{\tau_{\lambda}a: \lambda \in \R,
a=\nu_0+\nu_1s+\dots +\nu_ns^n,  \nu_0,\dots ,\nu_n \in \R,  n \in \N \}.
\]
In view of (\ref{ex5}),
\[
x_1=\frac {y*\psi} {\varphi}-a_1=a_2*\varphi+\dots +a_k*\varphi^{k-1}+\gamma*\varphi^{k-1}
\]
is a function such that $\supp x_1 \subset [-k+1,k-1]$.
Similarly,
\[
x_2=\frac {x_1} {\varphi}-a_2=a_3*\varphi+\dots +a_k*\varphi^{k-2}+\gamma*\varphi^{k-2}
\]
is a function such that $\supp x_2 \subset [-k+2,k+2]$. Repeating
this $k$ times we get $\supp \gamma \subset \{0\}$, so $\gamma=0.
\ $ Hence $y*\psi_2=0$. Consequently, $\psi_2=0$ and $\
\psi=\psi_1 \in E^k$. This implies (\ref{ex3}), which was to be
shown.
\end{example}

\end{document}